\documentclass{article}

\usepackage[latin1]{inputenc}
\usepackage{amsthm}
\usepackage{amsmath}
\usepackage{amsfonts}
\usepackage{amssymb}
\usepackage{enumerate}
\usepackage[margin=1in]{geometry}
\usepackage[usenames]{color}
\usepackage{tikz,hyperref}
\usetikzlibrary{graphs,graphs.standard,quotes}

\newtheorem{lem}{Lemma}
\newtheorem{cor}[lem]{Corollary}
\newtheorem{thm}[lem]{Theorem}

\newtheorem{conj}{Conjecture}[section]

\numberwithin{equation}{section}

\newcommand{\ben}{\begin{enumerate}}
\newcommand{\een}{\end{enumerate}}

\newcommand{\bfi}{\begin{figure} \begin{center}}
\newcommand{\efi}{\end{center} \end{figure}}
\newcommand{\capt}{\caption}
\newcommand{\hs}[1]{\hspace{#1}}

\newcommand{\da}{\hs{-2pt}\downarrow}

\author{Kenneth Barrese}
\title{A Graph Theory of Rook Placements}

\begin{document}
\maketitle

\begin{abstract}

Two boards are rook equivalent if they have the same number of non-attacking rook placements for any number of rooks. Define a rook equivalence graph on an equivalence class of Ferrers boards by specifying that two boards are connected by an edge if you can obtain one of the boards by moving squares in the other board out of one column and into a single other column. Given such a graph, we characterize which boards will yield connected graphs. We also provide some cases where common graphs will or will not be the graph for some set of rook equivalent Ferrers boards. Finally, we extend this graph definition to the $m$-level rook placement generalization developed by Briggs and Remmel. This yields a graph on the set of rook equivalent singleton boards, and we characterize which singleton boards give rise to a connected graph.

\end{abstract}

%


\section{Introduction}
\label{sec:in}

Broadly speaking, rook theory is the study of how many ways a number of rooks may be placed on a board, which is a collection of square cells grouped into a grid of rows and columns. The formal study of rook theory began with Kaplansky and Riordan~\cite{kr:pra}, who connected rook placements with elements of the symmetric groups, $S_n$, and used rook placements to study permutations with various restrictions.

Two boards are equivalent if there are the same number of ways to place $k$ rooks on both boards, for any non-negative integer $k$. The number of ways to place $k$ rooks on a board is also called the $k$-th rook number of the board. Foata and Sch\"{u}tzenberger~\cite{fs:rpf} demonstrated that for Ferrers boards, a convenient subset of boards, any rook equivalence class would contain a unique element of a specified type. They also defined explicit bijections between rook placements of $k$ rooks on any two rook equivalent Ferrers boards. Their result was primarily geometric, and involved transposing specified subboards of a Ferrers board.

Goldman, Joichi, and White~\cite{gjw:rtI} connected the results of Foata and Sch\"{u}tzenberger to algebraic combinatorics. They did so by slightly adjusting the existing definition of the rook polynomial, a generating function for the rook numbers. By redefining it to be in the falling factorial basis of polynomials, rather than the standard basis, they were able to factor the rook polynomial of a Ferrers board. This led them to a much simpler proof that Foata and Sch\"{u}tzenberger's unique element existed in any equivalence class, as well as allowing them to enumerate the size of the rook equivalence classes of Ferrers boards.

Briggs and Remmel~\cite{br:mrn} gave a generalization of rook placements by associating sets of $m$ consecutive rows into a single level. By doing so, they created rook placements that correspond to the wreath product $C_m\wr S_n$, where $C_m$ is the cyclic group with $m$ elements, in the way that ordinary rook placements correspond to $S_n$. Working on a subset of Ferrers boards that behave nicely with their $m$-level rook placements, called singleton boards, Briggs and Remmel were able to extend the factorization theorem of Goldman, Joichi, and White to their generalization.

Through a collaboration between Loehr, Remmel, Sagan, and myself~\cite{blrs:mrp}, we factored the $m$-level rook polynomial it a couple different ways. One method considered the heights of the different columns, similar to what Goldman, Joichi, and White and Briggs and Remmel have done. The other method focused on the number of cells in each level. In a second paper~\cite{blrs:bom} we were able to extend the bijections that Foata and Sch\"{u}tzenberger developed to $m$-level rook equivalence classes of singleton boards.

In this paper we connect rook theory and graph theory by defining a graph whose vertices are the Ferrers boards in a rook equivalence class, and whose edges are simple geometric alterations that transform one board to another in the equivalence class. In order to keep this paper self contained, Section~\ref{sec:gt} provides a very brief introduction to some definitions and concepts from graph theory. In Section~\ref{sec:rp} we provide a look at the background from rook theory. In Section~\ref{sec:gtrp} we define the rook equivalence graph, and work towards a criterion for whether this graph is connected or disconnected. The focus of Section~\ref{sec:commongraphs} is demonstrating that the complete graphs are obtained as the rook equivalence graphs for certain Ferrers boards while some small bipartite graphs do not exist as rook equivalence graphs.

We move to the $m$-level rook placement generalization in Section~\ref{sec:mrp}. This facilitates a generalized $m$-level rook equivalence graph introduced in Section~\ref{sec:mgtrp}, which once again concludes with a criterion for when the $m$-level rook equivalence graph is connected. Finally, in Section~\ref{sec:fproj} we provide some open conjectures building upon the new definition of rook equivalence graphs in this paper. The first conjecture concerns the non-existence of bipartite rook equivalence graphs, motivated by the results in Section~\ref{sec:commongraphs}. The second conjecture hypothesizes the existence of explicit bijections, in the spirit of Foata and Sch\"{u}tzenberger~\cite{fs:rpf}, induced by the edges of the rook equivalence graph.

\section{Introduction to Graph Theory}
\label{sec:gt}

Presented in this section is a short introduction to the concepts of graph theory needed to keep this paper internally complete. Anyone already familiar with graph theory can probably skip this section.

A \emph{graph} is a pair of sets, a set of vertices $V = \{v_1,\ldots,v_n\}$ and a set of edges $E =\{e_1,\ldots,v_M\}$, where each edge is an unordered pair of distinct vertices, $e = \{v_i,v_j\}$. Visually, graphs are usually represented by assigning vertices to unique points in space, and then drawing a line between vertices $v_i$ and $v_j$ if and only if $e = \{v_i,v_j\}$ is an element of the edge set. A graph is called \emph{connected} if, given any two vertices $v_i,v_j \in V$, there exists a sequence of edges $e_1,e_2,\ldots,e_p\in E$ such that $v_i \in e_1$, $v_j \in e_p$ and $e_n \cap e_{n+1} \neq \varnothing$ for all $1\leq n < p$. Figure~\ref{graphex} contains three examples of graphs, the first is not connected, but the other two are.

\bfi
\resizebox{100pt}{!}{
\begin{tikzpicture}
  [scale=.8,auto=left,every node/.style={circle,draw=black}]
  \node (n4) at (0,4)  {4};
  \node (n5) at (4,4)  {5};
  \node (n1) at (-2,0) {1};
  \node (n2) at (4,0)  {3};
  \node (n3) at (0,0)  {2};

  \foreach \from/\to in {n4/n2,n2/n5,n2/n3,n3/n4}
    \draw (\from) -- (\to);

\end{tikzpicture}
}
\hspace{30pt}
\resizebox{100pt}{!}{
\begin{tikzpicture}
  \graph[circular placement, radius=3cm,
         empty nodes, nodes={circle,draw=black}] {
    \foreach \x in {a,...,f} {
      \foreach \y in {\x,...,f} {
        \x -- \y;
      };
    };
  };
\end{tikzpicture}
}
\hs{30pt}
\resizebox{100pt}{!}{
\begin{tikzpicture}
  [scale=.8,auto=left,every node/.style={circle,draw=black}]
  \node (n3) at (0,4)  {3};
  \node (n5) at (4,4)  {b};
  \node (n4) at (4,0)  {a};
  \node (n1) at (0,0)  {1};
  \node (n2) at (0,2) {2};

  \foreach \from/\to in {n4/n2,n4/n1,n1/n5,n2/n5,n5/n3,n3/n4}
    \draw (\from) -- (\to);

\end{tikzpicture}
}
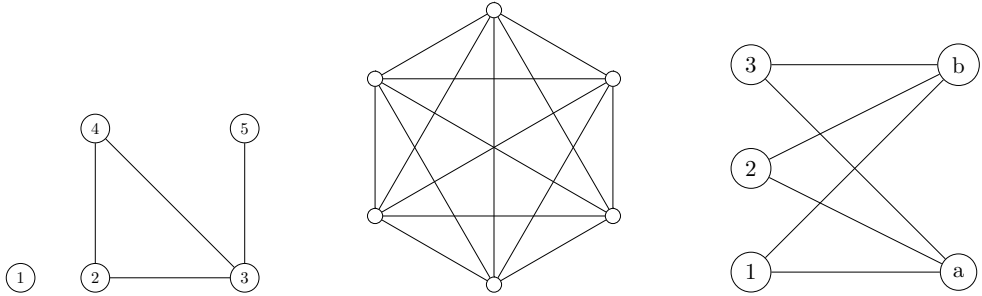
\capt{\label{graphex} On the left, a graph with five vertices and six edges. In the middle, $K_6$, the complete graph on six vertices. On the right, $K_{3,2}$, a complete bipartite graph.}
\efi

Some graphs of particular note are the complete graphs and the complete bipartite graphs. A \emph{complete graph} on $n$ vertices, denoted $K_n$, has a vertex set of size $n$ and edge set of size ${n \choose 2}$, requiring that every vertex be connected to every other vertex. A graph is \emph{bipartite} if its vertex set can be decomposed into two disjoint subsets $V = V_1 \sqcup V_2$ such that every edge in the edge set contains one vertex from $V_1$ and one vertex from $V_2$. A \emph{complete bipartite graph}, denoted $K_{n,m}$, is a graph where $V = V_1 \sqcup V_2$ such that $|V_1| = n$, $|V_2|=m$, and the edge set is of size $n\cdot m$, containing every possible edge with one vertex from $V_1$ and one vertex from $V_2$. The middle graph of Figure~\ref{graphex} is a complete graph and the right graph is a complete bipartite graphs.

Two graphs, $G_1=\{V_1,E_1\}$ and $G_2 = \{V_2,E_2\}$ are called \emph{isomorphic} if there is a bijection $f$ between $V_1$ and $V_2$ such that $\{v_i,v_j\} \in E_1$ if and only if $\{f(v_i),f(v_j)\} \in E_2$. In other words, two graphs are isomorphic if they have the same number of vertices and their vertices are connected in the same way by edges.

\section{Rook Placements}
\label{sec:rp}

In this section we provide an introduction to the definitions and concepts central to rook theory. There is a decided emphasis on the algebraic methods pioneered by Goldman, Joichi, and White, but they should also be understood in the context of the more geometric reasoning of Foata and Sch\"{u}tzenberger, and the geometric interpretations which they have in their own right.

Begin by tiling the first quadrant with $1$ by $1$ square cells. A \emph{board} is any finite subset of this tiling. Like Foata and Sch\"{u}tzenberger we will restrict our attention to a convenient subset of boards, called Ferrers boards. A \emph{Ferrers board} is a board consisting of connected columns, each beginning in the bottom row of the first quadrant, where the column heights are non-decreasing from left to right. Consider Figure~\ref{Ferrers} for an example of a Ferrers board. We index Ferrers boards by their column heights, so the Ferrers board in Figure~\ref{Ferrers} is $B=(1,2,2,4)$. Note that the shape of a Ferrers board is unchanged if we add columns of height zero to the left side of the board, for example, $B=(1,2,2,4) = (0,0,1,2,2,4)$. Henceforth, any board we refer to should be assumed to be a Ferrers board and we will only specify Ferrers boards in Theorem statements as a reminder of this condition.

\bfi
\begin{tikzpicture}
\draw(-.5,1) node {$B=$};
\foreach \x in {0} 
   \draw (\x,0)--(\x,.5);
\foreach \x in {.5,1} 
   \draw (\x,0)--(\x,1);
\foreach \x in {1.5,2}
   \draw (\x,0)--(\x,2);
\foreach \y in {0,.5}
	\draw (0,\y)--(2,\y);
\foreach \y in {1}
        \draw (.5,\y)--(2,\y);
\foreach \y in {1.5,2}
          \draw(1.5,\y)--(2,\y);
\draw (.75, .25) node{$R$};
\draw (1.75, 1.25) node{$R$};
\end{tikzpicture}
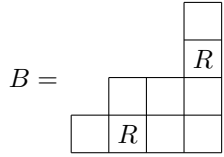
\capt{\label{Ferrers} The Ferrers board $B= (1,2,2,4)$}
\efi

Given a board, a \emph{rook placement of $k$ rooks} on that board is a selection of $k$ squares from that board, so that no two squares are in the same row or column. This corresponds to putting $k$ rook chess pieces, on the board so that no two rooks are attacking each other, since each rook attacks its own row and its own column. Figure~\ref{Ferrers} contains an example of a rook placement of $2$ rooks on the board. The rooks are in squares $(2,1)$ and $(4,3)$; note that we are using Cartesian coordinates to label the squares.

The \emph{$k$th rook number} of a board $B$ is the number of rook placements of $k$ rooks on that board and is denoted $r_k(B)$. For any board, $r_0(B) = 1$ and $r_1(B) = |B|$, the number of squares that $B$ contains. We call two boards \emph{rook equivalent} if they have the same rook numbers for all non-negative integers $n$, we denote this $B_1\equiv B_2$. In order to determine if two boards are rook equivalent, we define a  generating function for the rook numbers of a board. However, we do so in the falling factorial basis for the vector space of polynomials, which we define as follows. Given a non-negative integer $n$, we define \emph{$k$th falling factorial} of $x$ to be $$x\da_k = x(x-1)\cdots(x-(k-1)).$$ As with the traditional factorial, $x\da_0 = 1$. Thus $\{x\da_k\mid k\geq0\}$ is a basis for the vector space of polynomials.

Having the falling factorial basis, we can define the rook polynomial of a board. Given a board $B=(b_1,b_2,\ldots,b_n)$ with $n$ columns, the \emph{rook polynomial} of $B$ is $$p(B,x) = \sum_{k=0}^nr_k(B)x\da_{n-k}.$$ Notice first that this is the generating function for the rook numbers in the falling factorial basis. Notice also that, while adding columns of height zero to the left of the board $B$ will not change the geometry of $B$, and thus not affect the rook numbers of the board, it will affect the rook polynomial of the board. This is because the polynomial is indexed from zero to the number of columns of $B$. Therefore, we can conclude that two boards are rook equivalent if and only if they have the same rook polynomial as long as we pad the boards with columns of height zero until they both have the same number of columns.

Goldman, Joichi, and White obtained a beautiful theorem factoring the rook polynomial of $B$ in terms of the column heights of $B$ as follows:
\begin{thm}[Goldman-Joichi-White~\cite{gjw:rtI}]
\label{gjw1}
If $B=(b_1,\dots,b_n)$ is a Ferrers board, then
$$
\sum_{k=0}^n r_{k}(B) x\da_{n-k} = \prod_{i=1}^n (x+b_i-(i-1)).
$$
\end{thm}
In particular, we know all $n$ roots of $p(B,x)$. This motivates us to define the \emph{root vector} of $B$ as a vector containing all the roots ordered so the $i$th entry of the root vector is the root that comes from the $i$th column of $B$. In other words, if $B= (b_1,b_2,\ldots,b_n)$ then the root vector of $B$ is $\xi(B)=\langle 0-b_1,1-b_2,\ldots,(n-1)-b_n\rangle $. For the board $B=(1,2,2,4)$ in Figure~\ref{Ferrers}, $\xi(B) = \langle0-1,1-2,2-2,3-4\rangle = \langle-1,-1,0,-1\rangle$. The root vector gives another characterization of rook equivalence.

\begin{cor}[Goldman-Joichi-White~\cite{gjw:rtI}]
\label{rearrange}
For two Ferrers boards $B_1$ and $B_2$, $B_1 \equiv B_2$ if and only if $\xi(B_2)$ is a rearrangement of the elements of $\xi(B_1)$.
\end{cor}


Given a root vector, we can reconstruct the board $B$ for which it is $\xi(B)$ since the $i$th entry is $\xi_i = (i-1) - b_i$. Goldman, Joichi, and White use this fact to give a root vector proof of a theorem previously proved by Foata and Sch\"{u}tzenberger.

\begin{thm}[Foata-Sch\"{u}tzenberger~\cite{fs:rpf}]
\label{unique}
Every Ferrers board is rook equivalent to a unique Ferrers board with strictly increasing column heights.
\end{thm}

The proof that Goldman, Joichi, and White gave for this result relied on the fact that a strictly increasing board will have a root vector which begins by increasing from zero to its maximum value and thereafter weakly decreases. The increasing beginning corresponds to columns of height zero on the left side. The entries past the first element of maximum value correspond to non-empty columns. Remember that the $i$th entry of the root vector is $(i-1)-b_i$. If $\xi_{i+1} \leq \xi_i$, that means $b_i +1 \leq b_{i+1}$, which implies that the column heights are strictly increasing.

 However, not all rearrangements of $\xi(B)$ will necessarily correspond to Ferrers boards. Goldman, Joichi, and White give a criterion for when a vector is the root vector of some board.
\begin{thm}[\cite{gjw:rtI}]
\label{gjw2}
A vector $\xi = \langle\xi_1,\xi_2,\ldots,\xi_n\rangle$ is the root vector of some Ferrers board if and only if
\ben
   \item[(i)] $\xi_1 \leq 0$
   \item[(ii)] $\xi_i +1 \geq \xi_{i+1} $
\een
\end{thm}

Theorem~\ref{gjw2} is of great importance is Section~\ref{sec:gtrp}, so let us consider some consequences of it. Firstly, if a rook vector consists only of non-negative entries, then condition (i) requires that the first entry must have value $0$. The second condition specifies that an entry in the root vector never increases by more than $1$ over the value directly to its left. This implies that the leftmost occurrence of any positive value in the root vector must be proceeded by an entry of value exactly one less, since there is no way to ``jump over'' any positive value while only increasing the value of elements by $1$ going from left to right.

By counting the number of ways to reorder $\xi(B)$ so that the resulting vector is still the root vector for a board, Goldman, Joichi, and White were able to enumerate the set of boards rook equivalent to $B$. In order to count reordering, given a root vector $\xi(B)$ of a board with $n$ columns, let us define non-negative integers $$v_i =\left |\{k\mid1\leq k \leq n \text{ and } \xi(B)_k =i\}\right |.$$ In other words, $v_i$ is the number of times $i$ occurs in $\xi(B)$.
\begin{thm}[Goldman-Joichi-White~\cite{gjw:rtI}]
\label{gjw3}
Let $B=(b_1,\ldots,b_n)$ be a Ferrers board padded with sufficient columns of height zero on the left so that all entries of $\xi(B)$ are non-negative. The number of Ferrers boards rook equivalent to $B$ is: $$\prod_{i\geq0}{v_i+v_{i+1}-1 \choose v_{i+1}}.$$
\end{thm}

\section{Graph Theory of Rook Placements}
\label{sec:gtrp}

This section defines the rook equivalence graph motivated by geometric manipulation of the shape of rook equivalent Ferrers boards. We then segue into considering not the geometry of the boards, but their algebra, in the form of their root vectors. This enables us to develop a criterion for when a rook equivalence graph is connected.

In order to connect rook theory to graph theory, we define the \emph{rook equivalence graph} as follows. Given a board $B$, the vertex set $V$ will be the set of all boards rook equivalent to $B$, including $B$ itself. The edge set consists of pairs of boards $\{B_1,B_2\}$ such that $B_1$ and $B_2$ differ in only two columns $i$ and $j$ where $B_1$ has $k$ more squares in column $i$ than $B_2$ and $k$ fewer squares in column $j$ than $B_2$. Graphically, we can consider transforming $B_1$ into $B_2$ by taking $k$ squares from the top of column $i$ and moving them to the top of column $j$. Figure~\ref{edge} illustrates this with a pair of boards connected by an edge. Both root vectors are shown to demonstrate that $B_1 \equiv B_2$. We denote the root equivalence graph of the rook equivalence class containing $B$ by $G(B)$.

\bfi
\begin{tikzpicture}
\draw(-.5,1) node {$B_1=$};
\foreach \x in {0} 
   \draw (\x,0)--(\x,.5);
\foreach \x in {.5,1} 
   \draw (\x,0)--(\x,1);
\foreach \x in {1.5,2}
   \draw (\x,0)--(\x,2);
\foreach \y in {0,.5}
	\draw (0,\y)--(2,\y);
\foreach \y in {1}
        \draw (.5,\y)--(2,\y);
\foreach \y in {1.5,2}
          \draw(1.5,\y)--(2,\y);
\draw [fill=lightgray] (0,0) rectangle (.5,.5);
\end{tikzpicture}
\hs{40pt}
\begin{tikzpicture}
\draw(-.5,1) node {$B_2=$};
\draw (0,0)--(.5,0);
\foreach \x in {.5} 
   \draw (\x,0)--(\x,1);
\foreach \x in {1}
   \draw (\x,0)--(\x,1.5);
\foreach \x in {1.5,2}
   \draw (\x,0)--(\x,2);
\foreach \y in {0,.5,1}
	\draw (.5,\y)--(2,\y);
\foreach \y in {1.5}
        \draw (1,\y)--(2,\y);
\foreach \y in {2}
          \draw(1.5,\y)--(2,\y);
\draw [fill=lightgray] (1,1) rectangle (1.5,1.5);
\end{tikzpicture}
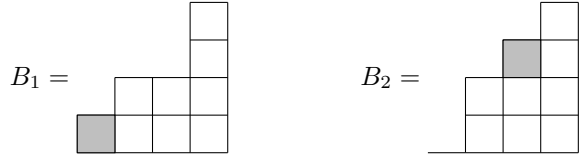
\capt{\label{edge} $B_1= (1,2,2,4)$, $\xi(B_1)=\langle-1,-1,0,-1\rangle$. $B_2=(0,2,3,4)$, $\xi(B_2)=\langle0,-1,-1,-1\rangle$.}
\efi

Consider that $B_1$ and $B_2$ in Figure~\ref{edge} differ only by moving squares, the one shaded square in this case, from a column of $B_1$ to another column in $B_2$. Remember that we must consider $B_2$ as having four columns, in order to compare it to $B_1$ which also has four columns. Finally, examine how the two root vectors differ only in positions $1$ and $3$. where the values have been swapped, this motivates the next theorem.

\begin{thm}
\label{rootswap}
If $B_1 \equiv B_2$, both containing $n$ columns, are such that $\{B_1,B_2\}\in E$, then the root vectors $\xi(B_1)$ and $\xi(B_2)$ differ only in two positions, which are swapped. That is, if $\xi(B_1) = \langle z_1,\ldots,z_{i_1},\ldots,z_{i_2},\ldots,z_n\rangle$ then $\xi(B_2)=\langle z_1,\ldots,z_{i_2},\ldots,z_{i_1},\ldots,z_n\rangle$ for some $1< i_1 < i_2 \leq n$ where $z_{i_1} \neq z_{i_2}$.
\end{thm}

\begin{proof}
Without loss of generality, let us assume that the squares which change columns between $B_1$ and $B_2$ are further to the left in $B_1$ than in $B_2$. By the definition of $\{B_1,B_2\} \in E$, if $B_1=(b_1,\ldots,b_{i_1},\ldots,b_{i_2},\ldots,b_n)$ then $B_2 = (b_1,\ldots,b_{i_1}-k,\ldots,b_{i_2}+k,\ldots,b_n)$ for some integer $k>0$. Then $\xi(B_1) = \langle-b_1,\ldots,(i_1-1)-b_{i_1},\ldots,(i_2-1)-b_{i_2},\ldots,(n-1)-b_n\rangle$ and $\xi(B_2) = \langle-b_1,\ldots,(i_1-1)-b_{i_1}+k,\ldots,(i_2-1)-b_{i_2}-k,\ldots,(n-1)-b_n\rangle$.

However, $\{B_1,B_2\} \in E$ implies that $B_1 \equiv B_2$ which further implies that $\xi(B_2)$ is a rearrangement of the elements of $\xi(B_1)$. Since $\xi(B_1)$ and $\xi(B_2)$ only differ in two positions, it must be the case that $(i_1-1)-b_{i_1} = (i_2-1)-b_{i_2}-k$ and $(i_2-1)-b_{i_2} = (i_1-1)-b_{i_1}+k$. Therefore the vectors $\xi(B_1)$ and $\xi(B_2)$ are identical, except for in two indexes where their elements are swapped.
\end{proof}

While the edges of the rook equivalence graph are originally motivated by the geometric notion of moving squares from one column to another, henceforth we will usually show that two boards are connected by an edge if their root vectors differ by two swapped elements. We will return to the geometric motivation again in Section~\ref{sec:fproj}. For now, having characterized when $B_1$ and $B_2$ are connected by an edge in terms of their respective root vectors, we shall develop a criterion for when the rook equivalence graph of the equivalence class containing $B$ is connected.

\begin{lem}
\label{checktwice}
Suppose $\xi(B)=\langle z_1,\ldots,z_n\rangle$ is the root vector for some Ferrers board. Let $\xi'$ be equivalent to $\xi(B)$ except in positions $1<i_1$ and $1<i_2$ which are swapped, and further assume that $z_{i_1}>z_{i_2}$. If $z_{i_2-1}+1 \geq z_{i_1}$ and $z_{i_2}+1 \geq z_{i_1+1}$ we may conclude that $\xi'$ is the root vector for some Ferrers board.
\end{lem}

\begin{proof}
Since we do not affect the first element of the root vector, we need only check that the second condition of Theorem~\ref{gjw2} is satisfied for all consecutive elements in $\xi'$. However, since the condition was already true for all pairs of neighbors in $\xi(B)$, we need only check the pairs that change in $\xi'$, specifically the elements of $\xi'$ in positions $i_1-1$ and $i_1$, positions $i_1$ and $i_1+1$, positions $i_2-1$ and $i_2$, and in positions $i_2$ and $i_2+1$. Since $\xi(B)$ is a root vector, we know $z_{i_1-1} +1 \geq z_{i_1}$, and since $z_{i_1} > z_{i_2}$ we must have $z_{i_1-1} + 1 \geq z_{i_2}$ as desired since those are the elements in positions $i_1-1$ and $i_1$ after the swap.

Similarly, we know that $z_{i_2}+1 \geq z_{i_2+1}$ since $\xi(B)$ is a root vector, so we must have $z_{i_1} + 1\geq z_{i_2+1}$. Thus the elements in positions $i_2$ and $i_2+1$ after the swap satisfy the second condition. All that is left is to check that the new element in position $i_1$ with the element to its right, and the new element in position $i_2$ with the element to its left. Thus $z_{i_2} +1 \geq z_{i_1+1}$ and $z_{i_2-1}+1\geq z_{i_1}$ together are equivalent to $\xi'$ being the root vector for some Ferrers board.
\end{proof}

Lemma~\ref{checktwice} enables us to examine only two pairs of neighbors to determine if a vector that results from swapping two elements of an existing root vector is still a root vector. Specifically, we must check the neighbor to the left of the element that increased in value, and the neighbor to the right of the element that decreased in value. We will use this to streamline the proof of the following lemma, which orders the right side of a root vector in weakly decreasing order by repeatedly increasing the leftmost element on the right side that is not already in weakly decreasing order. Figure~\ref{decfig} contains a short example of this process.

\begin{lem}
\label{decreasing}
Let $B$ be a Ferrers board with $n$ columns that has $M$ as the maximum element in its root vector. If the first appearance of $M$ in $\xi(B)$ is in position $j$, then vertex $B$ is connected by a sequence of edges in $G(B)$ to another Ferrers board $B'$ which has a root vector $\xi(B')$ identical to $\xi(B)$ in the first $j$ positions, and weakly decreasing for indices $i \geq j$.
\end{lem}

\begin{proof}
We will prove this by swapping elements of $\xi(B)$, in such a way that the vector remains a root vector for some board, until we arrive at the root vector $\xi(B')$. Theorem~\ref{gjw2} gives a set of conditions for when a vector is a root vector. As long as we do not alter the first element of the root vector, the first condition will hold, so all that we must consider is that the second condition, $\xi_{i+1} \leq \xi_i+1$, holds after each swap we make. As noted in Lemma~\ref{checktwice}, we only need to check two pairs of elements after every swap.

Since we have assumed that $z_j$ is the leftmost instance of the maximum element of $\xi(B)$, we already know that $z_j \geq z_i$ for all $i > j$. Next we want to make sure that $z_{j+1} \geq z_i$ for all $i > j+1$. If $z_{j+1}$ is not greater than some element further to the right than it, swap the $z_{j+1}$ and the rightmost element with a greater value than $z_{j+1}$. Let us denote the new position of element $z_{j+1}$ as index $k$ and thus the new element in position $j$ by $z_k$ since it was previously in position $k$ . Now let us check that the second condition of Theorem~\ref{gjw2} holds for the two pairs of neighboring elements that Lemma~\ref{checktwice} specifies. Since $z_k>z_{j+1}$ by assumption, we need to check that $z_j +1 \geq z_k$ and $z_{j+1} +1 \geq z_{k+1}$.

By assumption, $z_j$ is greater than or equal to all elements to its right because it is a maximum element of $\xi(B)$, so $z_j +1 \geq z_k$ must hold for this pair. Also, since we swapped $z_{j+1}$ with the rightmost element greater than it, we know that $z_{j+1} \geq z_{k+1}$ and thus that $z_{j+1}+1 \geq z_{k+1}$. Lemma~\ref{checktwice} allows us to conclude that this rearranged vector is the root vector for some board, and from Theorem~\ref{rootswap} we know that the new board and $B$ are connected by an edge in the rook equivalence graph.

Repeat this process for the new element in position $j+1$. Since the root vector is finite, at some point the process must terminate, which can only happen when the element currently in position $j+1$, which we will now call $z'_{j+1}$, is greater than or equal to all the elements to its right in the current root vector. The vector we obtain must be the root vector of some board, and that board will be connected by a sequence of edges back to $B$, the original board. Note that in the process of adjusting the element in the $j+1$st position, the only property of $z_j$ being a maximum element we used is that it was greater than or equal to all elements to its right, which is now also true of $z'_{j+1}$.

This allows us to repeat this process on $z_{j+2}$, eventually obtaining $z'_{j+2}$ which is greater than or equal to all elements to its right. Furthermore, the board we obtain at this point will be connected by a sequence of edges to the board which contained the original $z_{j+2}$ in the $j+2$nd position, and thus will be connected back to the original board $B$. Inducting on the index of the position being considered, we eventually obtain a board $B'$, connected by a sequence of edges back to $B$, such that for every $z'_i$ with $i \geq j$ in $\xi(B')$, $z'_i$ is greater than or equal to every element to its right in $\xi(B')$. In particular, $z'_i \geq z'_{i+1}$ so the elements at or to the right of position $j$ are in weakly decreasing order.
\end{proof}

\bfi
$\langle0,1,1,2,1,2,3,4,2,2,3,4,1,2,3\rangle$

$\langle0,1,1,2,1,2,3,4,{\bf 3},2,3,4,1,2,{\bf 2}\rangle$

$\langle0,1,1,2,1,2,3,4,{\bf 4},2,3,{\bf 3},1,2,2\rangle$

$\langle0,1,1,2,1,2,3,4,4,{\bf 3},3,{\bf 2},1,2,2\rangle$

$\langle0,1,1,2,1,2,3,4,4,3,3,2,{\bf 2},2,{\bf 1}\rangle$
\capt{\label{decfig} Swapping pairs of elements until the root vector is weakly decreasing past its maximum element. The pair of bold elements are the two elements swapped in each step.}
\efi

Consider Figure~\ref{decfig} for a short example of the sequence of element swaps that transforms $\xi(B)$ to $\xi(B')$ as described in Lemma~\ref{decreasing}. Since we can order the elements to the right of the maximum element of $\xi(B)$ in weakly decreasing order while remaining in a connected component of the rook equivalence graph, our next lemma tackles the task of shifting the maximum element to the left as far as possible. We do this by finding the largest element that occurs at least twice to the left of the maximum element, and swapping elements until the size of the largest element occurring twice to the left of the maximum element decreases. This can be accomplished as long as every value between that of the element to be adjusted and the value of the maximum element occurs in the root vector at least twice. Figure~\ref{mleft} gives an example of this process.

\begin{lem}
\label{maxleft}
Let $B$ be a Ferrers board with $n$ columns, padded with columns of height zero on the left so that $\xi(B)=\langle z_1,\ldots,z_n\rangle$ contains only non-negative entries. Remember $v_k$ is the number of entries of the root vector of value $k$. Let $M$ be the greatest integer such that $v_M \neq 0$. If $v_i >1$ implies that $v_{i+1} >1$ for $0\leq i < M-1$, then $B$ is connected by a sequence of edges in $G(B)$ to another Ferrers board $B'$ which has a root vector $\xi(B') = \langle0,1,\ldots,M-1,M,\ldots,z_n\rangle$. That is to say, the first $M+1$ entries of $\xi(B')$ are the integers from $0$ to $M$ in increasing order.
\end{lem}

\begin{proof}
Assume that the leftmost occurrence of the maximum element $M$ occurs at index $j$, so that $M=z_j$. Since every positive value must occur for the first time in the root vector directly to the right of an element of value one smaller, elements of every value from $0$ to $M-1$ occur at least once to the left of index $j$. Lemma~\ref{decreasing} allows us to assume that $\xi(B)$ is non-increasing in positions $i \geq j$ without loss of generality we. Let $o$ be the largest value in $\xi(B)$ such that $o$ occurs in at least two positions to the left of $z_j$. Let the leftmost occurrence of element $o$ in $\xi(B)$ be at position $k<j$.


Begin by swapping the elements in the $k+1$st position with the rightmost, and therefore smallest, element $z_i$ with $i \geq j$ such that $z_i >z_{k+1}$. We are guaranteed such an element exists because $o<M = z_j$. As before, Lemma~\ref{checktwice} specifies which two pairs of neighbors we should check to determine that the resulting vector is the root vector of some board. We are assuming that $z_i > z_{k+1}$, so we need to check that $z_{k} + 1\geq z_i$ and $z_{k+1} +1 \geq z_{i+1}$. Since $o$ appears at least twice to the left of $z_j$, $e_o >1$, which implies that either $e_{o+1}>1$ or $o+1 = M$, if $o+1 \neq M$, then an element of value $o+1$ exists to the right of position $j$, since $o$ is the largest value appearing at least twice to the left of index $j$. Either way, we are guaranteed there is an element in or to the right of position $j$ with value $o+1$, since we replace $z_{k+1}$ with $z_i$, the smallest element larger than it with $i \geq j$, we know $z_k+1 \geq z_i$ but $z_i \leq o+1$. Since $z_k = o$, we get $z_k+1 \geq z_i$. We swapped $z_{k+1}$ with the rightmost element bigger than it, guaranteeing that $z_{i+1} \leq z_{k+1}$, and thus that $z_{k+1} +1 \geq z_{i+1}$. Once again, we may conclude that the resulting vector is the root vector of some board, by Lemma~\ref{checktwice}, and that the new board is connected by an edge to our old board in the rook equivalence graph by Theorem~\ref{rootswap}.

Because we swapped $z_{k+1}$ with the rightmost element with index $i \geq j$ greater than it, and since the elements with index $i \geq j$ were already in non-decreasing order, after performing this swap, the elements with index $i \geq j$ are still in non-decreasing order in the new root vector. That means that if element $o$ in position $k$ is still the leftmost instance of the largest element that occurs at least twice to the left of $z_j = M$, we can repeat this process to increase the value of $z_{k+1}$ yet again. The only way this could fail to hold true is if, after swapping, the new value in position $k+1$ is actually greater than $o$, which means it must be $o+1$. In this case, either there are two occurrences of value $o+1$ to the left of position $j$, the one that was already there and the one that started to the right of position $j$ and was swapped into position $k+1$, or $o+1 = M$ and the leftmost maximum element is now in position $k+1$.

If $o+1 \neq M$, then it is now the greatest value of an element that occurs twice to the left of position $j$, since it already occurred once to the left of $j$, then we swapped in another element of this value from the right of position $j$. In this case repeat the process from above to increase the value of the element in the $k+2$nd position. As long as $M$ is not the value to be swapped into a position, we can iterate this process to put a value one greater in the position one to the right.

As soon as $o+1 = M$ we can no longer repeat this process, because $o+1$ is no longer a value that occurs twice to the left of $M$, it has become $M$. Since the root vector is finite, this will happen after a finite number of iterations. At this point we will have created a sequence of entries in our root vector where the $k$th entry is our original $o$, the $k+1$st entry is $o+1$, and so on until the $k+(M-o)$th entry is $M$. Specifically, since our original $o$ was the leftmost occurrence of value $o$ to the left of the first occurrence of value $M$, now value $o$ only occurs once to the left of the first occurrence of value $M$. Therefore we have either decreased the largest value that occurs twice to the left of the first occurrence of $M$, or now no value occurs twice to the left of the first occurrence of $M$.

Since the root vector is finite and each application of this processes moves the first occurrence of $M$ to the left, after repeated applications of this process, eventually there will be no value that occurs twice to the left of the first occurrence of $M$ in the root vector. At this point, the first $M+1$ elements of the root vector must be $\xi(B') = \langle0,1,\ldots,M-1,M,\ldots\rangle$. Because each swap we performed corresponded to an edge in the rook equivalence graph, the final board $B'$ is connected to the original board $B$ by a sequence of edges, completing the proof.
\end{proof}

\bfi
$\langle0,1,1,2,1,2,3,4,4,3,3,2,2,2,1\rangle$

$\langle0,1,1,2,2,2,3,4,4,3,3,2,2,1,1\rangle$

$\langle0,1,1,2,3,2,3,4,4,3,2,2,2,1,1\rangle$

$\langle0,1,1,2,3,3,3,4,4,2,2,2,2,1,1\rangle$

$\langle0,1,1,2,3,4,3,4,3,2,2,2,2,1,1\rangle$
\capt{\label{mleft} Swapping pairs of elements until the greatest element occurring twice to the left of the first $4$ decreases from value $2$ to value $1$.}
\efi

Figure~\ref{mleft} illustrates a step process described in Lemma~\ref{maxleft}. Note that in the first vector $2$ is the largest value occurring twice to the left of the first $4$, but by the final vector $1$ is the only value occurring twice to the left of the first $4$. At that point we could use the guarantee of Lemma~\ref{decreasing} to reorder the elements to the right of the first $4$ in non-increasing order, then reapply the process to reduce the number of elements of value $1$ left of the first $4$ to only one.

With Lemma~\ref{decreasing} and Lemma~\ref{maxleft}, we are ready to prove the major theorem of this section, which determines whether the rook equivalence graph $G(B)$ is connected based on the elements of $\xi(B)$.

\begin{thm}
\label{connected}
Let $B$ be a Ferrers board with $n$ columns, padded with columns of height zero on the left so that $\xi(B)=\langle z_1,\ldots,z_n\rangle$ contains only non-negative entries. As before, let $M$ be the greatest integer such that $v_M \neq 0$. $G(B)$ is connected if and only if $v_i >1$ implies that $v_{i+1} >1$ for $0\leq i < m-1$.
\end{thm}

\begin{proof}
The proof that the rook equivalence graph is connected if $v_i>1$ implies $v_{i+1}>1$ for $0\leq i<m-1$ is basically complete given Lemma~\ref{decreasing}, Lemma~\ref{maxleft}, and Theorem~\ref{unique}. By Lemma~\ref{maxleft}, we can find a sequence of edges from $B$ to a rook equivalent board where the rook vector entries begin with the sequence $0,1,\ldots,M-1,M$. Lemma~\ref{decreasing} guarantees a sequence of edges from that board to one where the entries past the first occurrence of $M$ are non-increasing. Finally, the Goldman, Joichi, and White proof of Theorem~\ref{unique} implies that the board with this root vector is the unique board in the equivalence class with increasing column heights. Since an arbitrary board in the equivalence class is connected by a sequence of edges to the unique representative of the equivalence class, any two boards in the equivalence class can be connected by a sequence of edges.

To show the graph is connected only if these conditions hold, assume there is some value, $v<M$, that occurs exactly once in the root vector, in position $i$, but there is a value less than $v$ that occur at least twice. Since the entries of the root vector can increase by at most $1$ from entry $i$ to entry $i+1$, and root vectors with non-negative entries always have $0$ as their leftmost entry. The first time any positive value occurs, it must be proceeded immediately to the left by the value that is one less. Therefore, since $v\geq 0$ occurs only once, and there is at least one entry with value greater than $v$, it must be the case that $z_{i+1}=v+1$.

These facts imply that, by swapping entries one pair at a time, we can never move $v$ out of position $i$ or $v+1$ out of position $i+1$, because the first occurrence of value $v+1$ must always be proceeded directly be the unique entry of value $v$ so both are fixed. But, by our assumption, there is some value $u<v$, that occurs twice in the root vector. We use this value to construct two root vectors that could never be connected by an edge path.

The first root vector we consider is the one guaranteed by Theorem~\ref{unique}, consisting of values that increase by one left to right, beginning with value $0$ in position $1$ to the maximum value $M$ in position $M+1$, then proceed in weakly decreasing order to the right of position $M+1$. In this vector value $v$ must be in position $v+1$. Compare this with the root vector which is exactly the same, but contains two elements of value $u$ next to each other in the otherwise strictly increasing sequence from $0$ to $M$. This is still the root vector of some board, because it still begins with an element of value $0$ and no element increases by more than $1$ over the value of the element directly to its left. However, in this root vector, the unique element $v$ is in position $v+2$. If we only alter the order of elements in our root vector by pairwise swaps, element $v$ is required to stay in that position, due to the leftmost value of $v+1$ directly to its right. Therefore the two root vectors described can never be connected by a sequence of edges, and the rook equivalence graph is disconnected.
\end{proof}

\section{Common Graphs as Rook Equivalence Graphs}
\label{sec:commongraphs}

Having defined the rook equivalence graph and presented some of its properties, it is worth considering which graphs are actually realized as rook equivalence graphs. In this section we demonstrate that the complete graphs are a subset of rook equivalence graphs. We also present some notable absences in the form of small, complete bipartite graphs.

\begin{thm}
\label{completeg}
For any integer $n>0$, $K_n$, the complete graph on $n$ vertices, is isomorphic to $G(B)$ where $B = (0,\ldots,n-1,n-1)$.
\end{thm}

\begin{proof}
Consider the root vector of $B=(0,\ldots,n-1,n-1)$. Since $\xi(B) = \langle0,\ldots,0,1\rangle$ must begin with an entry of value $0$, the single entry of value $1$ cannot occupy this position. However, since the root vector consists exclusively of elements of value $0$ or $1$, there is no way for any rearrangement of the root vector to have values that increase by more than $1$ from one element to the next. Thus, the element of value $1$ can appear in any of the other $n$ positions, including index $n+1$ where it begins. If the $1$ is in position $i>1$, it can be moved to position $j>1$ simply by swapping the elements in positions $i$ and $j$. Thus every one of the $n$ vertices of $G(B)$ is connected by an edge to every single other vertex of $G(B)$, and $G(B)$ is isomorphic to $K_n$.
\end{proof}
See Figure~\ref{completefig} for the four boards that form the four vertices of $K_4$ as per Theorem~\ref{completeg}.

\bfi
\begin{tikzpicture}
\draw(-.5,1) node {$B_1=$};
\foreach \x in {1} 
   \draw (\x,0)--(\x,1);
\foreach \x in {1.5}
   \draw (\x,0)--(\x,1.5);
\foreach \x in {2,2.5}
   \draw (\x,0)--(\x,2);
\foreach \y in {0}
	\draw (0,\y)--(2.5,\y);
\foreach \y in {.5,1}
        \draw (1,\y)--(2.5,\y);
\foreach \y in {1.5}
          \draw(1.5,\y)--(2.5,\y);
\foreach \y in {2}
         \draw(2,\y)--(2.5,\y);
\end{tikzpicture}
\hs{10pt}
\begin{tikzpicture}
\draw(-.5,1) node {$B_2=$};
\foreach \x in {.5,1} 
   \draw (\x,0)--(\x,.5);
\foreach \x in {1.5}
   \draw (\x,0)--(\x,1.5);
\foreach \x in {2,2.5}
   \draw (\x,0)--(\x,2);
\foreach \y in {0}
	\draw (0,\y)--(2.5,\y);
\foreach \y in {.5}
        \draw (.5,\y)--(2.5,\y);
\foreach \y in {1,1.5}
          \draw(1.5,\y)--(2.5,\y);
\foreach \y in {2}
         \draw(2,\y)--(2.5,\y);
\end{tikzpicture}
\hs{10pt}
\begin{tikzpicture}
\draw(-.5,1) node {$B_3=$};
\foreach \x in {.5} 
   \draw (\x,0)--(\x,.5);
\foreach \x in {1,1.5}
   \draw (\x,0)--(\x,1);
\foreach \x in {2,2.5}
   \draw (\x,0)--(\x,2);
\foreach \y in {0}
	\draw (0,\y)--(2.5,\y);
\foreach \y in {.5}
        \draw (.5,\y)--(2.5,\y);
\foreach \y in {1}
          \draw(1,\y)--(2.5,\y);
\foreach \y in {1.5,2}
         \draw(2,\y)--(2.5,\y);
\end{tikzpicture}
\hs{10pt}
\begin{tikzpicture}
\draw(-.5,1) node {$B_4=$};
\foreach \x in {.5} 
   \draw (\x,0)--(\x,.5);
\foreach \x in {1}
   \draw (\x,0)--(\x,1);
\foreach \x in {1.5,2,2.5}
   \draw (\x,0)--(\x,1.5);
\foreach \y in {0}
	\draw (0,\y)--(2.5,\y);
\foreach \y in {.5}
        \draw (.5,\y)--(2.5,\y);
\foreach \y in {1}
          \draw(1,\y)--(2.5,\y);
\foreach \y in {1.5}
         \draw(1.5,\y)--(2.5,\y);
\end{tikzpicture}
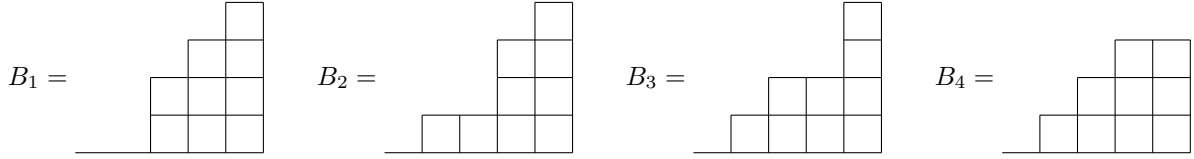
\capt{\label{completefig} The four vertices of $G(B_4)$, labeled so that $B_n$ has the $1$ in its root vector at index $n+1$. Note that $B_4$ is the board $B$ specified in the statement of Theorem~\ref{completeg}.}
\efi

\begin{thm}
\label{no22}
$K_{2,2}$ is not the rook equivalence graph for any Ferrers board.
\end{thm}

\begin{proof}
This proof proceeds by considering what boards could have the correct equivalence class size so that $K_{2,2}$ could be their rook equivalence graph. Since $K_{2,2}$ has four vertices, we need a board which has a rook equivalence class of size $4$. From Theorem~\ref{gjw3}, we know that we need $$4 = \prod_{i\geq 0}{v_i+v_{i+1}-1 \choose v_{i+1}}.$$ Furthermore, since $K_{2,2}$ is connected, we know that if $v_i = 1$, then either $i$ is the maximum value of elements in the root vector or $e_j=1$ for all $0\leq j \leq i$. For similar reasons, if $v_i \neq 1$, then $v_{i+1} \neq 1$ or $v_{i+1} = 1$ and $i+1$ is the maximum element of the root vector.

We can write $4$ as a product two ways, $4(\cdot 1)$ and $2 \cdot 2$. If ${v_i + v_{i+1} -1 \choose v_{i+1}}= 4$, then $v_i + v_{i+1} -1 = 4$ and $v_{i+1} \in \{1,4\}$. If $v_{i+1} = 1$, then $v_i = 4$, and we obtain a root vector looking like $\langle0,1,\ldots,i,i,i,i,i+1\rangle$. This root vector leads to a rook equivalence graph of $K_4$ by a reasoning similar to the proof of the previous theorem. Similarly, when $v_{i+1}=3$ and $v_i = 2$, the root vector has the form $\langle0,1,\ldots,i,i,(i+1),(i+1),(i+1)\rangle$. This root vector also leads to a rook equivalence graph isomorphic to $K_4$ since the second entry of value $i$ can be rearranged with the three entries of value $i+1$ freely.

This leaves us with the possibility that ${v_i + v_{i+1} -1 \choose v_{i+1}}= {v_{i+1} + v_{i+2} -1 \choose v_{i+2}}=2$. Note that we can conclude that the two instances have adjacent indexes because ${e_j + e_{j+1} -1 \choose e_{j+1}}=1$ implies $e_j = 1$, and we cannot have a value $j$ occurring once in our root vector between two other values that both occur more than once. However, ${v_i + v_{i+1} -1 \choose v_{i+1}}=2$ implies that $v_{i+1} = 1$ and $v_i = 2$, at which point $i+1$ must be the maximum value taken by elements of the root vector, in order for the rook equivalence graph to be connected. So there is no way that $v_{i+2}$ could be non-zero, and therefore no root vector of a board with a connected rook equivalence graph can fulfill  ${v_i + v_{i+1} -1 \choose v_{i+1}}= {v_{i+1} + v_{i+2} -1 \choose v_{i+2}}=2$. Thus there is no root vector that yields $K_{2,2}$ as its rook equivalence graph, as originally asserted.
\end{proof}

\begin{thm}
\label{no33}
$K_{3,3}$ is not the rook equivalence graph for any Ferrers board.
\end{thm}

\begin{proof}
The proof has a similar structure to the proof of Theorem~\ref{no22}. $K_{3,3}$ has $6$ vertices and we can write $6$ as a product two ways, $6$ and $2\cdot3$. If ${v_i + v_{i+1} -1 \choose v_{i+1}}=6$ then either $v_{i+1}=1$ and $v_i =6$, $v_{i+1} = 5$ and $v_{i} = 2$, or $v_{i+1}=2$ and $v_i = 3$. If the root vector satisfies either $v_{i+1}=1$ and $v_i =6$ or $v_{i+1} = 5$ and $v_{i} = 2$, the rook equivalence graph is $K_6$ as in the previous proof.

If $v_i = 3$ and $v_{i+1} = 2$, we get a root vector of the form $\langle0,0,0,1,1\rangle$. By inspection, this root vector leads to an equivalence class also containing $\langle0,0,1,0,1\rangle, \langle0,1,0,0,1\rangle, \langle0,1,1,0,0\rangle, \langle0,1,0,1,0\rangle$, and $\langle0,0,1,1,0\rangle$. However, any pair of the first three vertices listed are connected by an edge. This prevents the graph from being bipartite, because $\langle0,0,0,1,1\rangle$ and $\langle0,0,1,0,1\rangle$ are connected by an edge, so one must be in $V_1$ and the other in $V_2$. But then, since $\langle0,1,0,0,1\rangle$ is connected by an edge to each of them, it cannot be in either $V_1$ or $V_2$.

Lastly, we consider that $6 = 3\cdot 2$. As in the proof of the previous theorem, we saw that if  ${v_i + v_{i+1} -1 \choose v_{i+1}}=2$ it implies that $v_{i+1} = 1$ and $v_i = 2$, so $v_{i+2}$ must equal zero. However, if  ${v_i + v_{i+1} -1 \choose v_{i+1}}=3$ we have two options. If $v_{i+1} = 1$ and $v_i = 3$, we once again cannot have $v_{i+2} \neq 0$. On the other hand, if $v_{i+1} = 2$ and $v_i = 2$, then we can satisfy ${v_{i+1} + v_{i+2} -1 \choose v_{i+2}}=2$ by letting $v_{i+2} = 1$. This gives a root vector of the form $\langle0,0,1,1,2\rangle$.

If $\xi(B) = \langle0,0,1,1,2\rangle$ then the root vectors of the other boards in the rook equivalence class are $\langle0,1,0,1,2\rangle, \langle0,1,2,1,0\rangle, \langle0,1,1,2,0\rangle, \langle0,0,1,2,1\rangle,$ and $\langle0,1,2,0,1\rangle$. Unfortunately, the first five vectors listed form a cycle, so if the first vector is in $V_1$, the second must be in $V_2$, forcing the third to be back in $V_1$, and the fourth in $V_2$. Now the fifth vector is connected to the fourth, so it cannot be in $V_2$, but it also connects back to the first, so it cannot be in $V_1$. Ergo, this root vector does not lead to a rook equivalence graph that is $K_{3,3}$ although it is interesting and pictured in Figure~\ref{bothnot33}. Thus none of our possibilities lead to a graph that is $K_{3,3}$, so $K_{3,3}$ is not the rook equivalence graph for any rook equivalence class.
\end{proof}

\bfi
\begin{tikzpicture}
  [scale=.8,auto=left,every node/.style={circle,draw=black}]
  \node (n1) at (2,4)  {1};
  \node (n2) at (4,3)  {2};
  \node (n3) at (4,1) {3};
  \node (n4) at (2,0)  {4};
  \node (n5) at (0,1)  {5};
  \node (n6) at (0,3)  {6};

  \foreach \from/\to in {n1/n2,n1/n3,n1/n5,n1/n6,n2/n3,n2/n4,n2/n6,n3/n4,n3/n5,n4/n5,n4/n6,n5/n6}
    \draw (\from) -- (\to);

\end{tikzpicture}
\hs{30pt}
\begin{tikzpicture}
  [scale=.8,auto=left,every node/.style={circle,draw=black}]
  \node (n1) at (1.75,4)  {5};
  \node (n2) at (3.5,2.5)  {1};
  \node (n3) at (3.5,0) {2};
  \node (n4) at (1.75,2)  {6};
  \node (n5) at (0,0)  {3};
  \node (n6) at (0,2.5)  {4};

  \foreach \from/\to in {n1/n4,n1/n2,n1/n6,n2/n3,n3/n5,n5/n6}
    \draw (\from) -- (\to);

\end{tikzpicture}
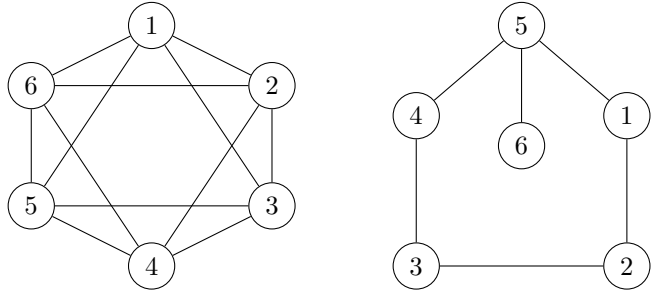
\capt{\label{bothnot33} The graph on the left is the rook equivalence graph for the equivalence class containing $B_1$ with $\xi(B_1) = \langle0,0,0,1,1\rangle$, notice that you can make a cycle containing $3$ vertices. The graph on the right comes from $B_2$ with $\xi(B_2) = \langle0,0,1,1,2\rangle$, notice that you can make a cycle containing $5$ vertices. The vertices of each graph are numbered in the order that the root vectors were listed in the proofs.}
\efi

\section{$m$-Level Rook Placements}
\label{sec:mrp}

Briggs and Remmel defined a generalization of rook placements called $m$-level rook placements~\cite{br:mrn}. This section introduces this generalization, and provides some important results that we will use to define and work with an $m$-rook equivalence graph.

As before, we will restrict our consideration to Ferrers boards. Given a board $B$ and a fixed integer $m>0$, we partition the rows of $B$ into sets of size $m$ called \emph{levels}, where the first level contains the bottom $m$ rows, the second level contains rows $m+1$ through $2m$, and so on until every row of $B$ is in a level. Note that the top level may contain some rows which do not contain squares of $B$.

An \emph{$m$-level rook placement of $k$ rooks} on $B$ is a subset of $B$ containing $k$ squares, no two of which are in the same level or column. Figure~\ref{mrook} has an example of a $2$-level rook placement of $3$ rooks on a board. Notice that an $m$-level rook placement replaces the role of a row with that of a level, which is a set of rows. Clearly then every $m$-level rook placement is also a rook placement, and a rook placement is equivalent to a $1$-level rook placement. The \emph{$k$th $m$-level rook number} of $B$, denoted $r_{k,m}(B)$, is the number of $m$-level rook placements of $k$ rooks on $B$. As before, we say two boards are \emph{$m$-level rook equivalent} if they have the same $m$-level rook numbers for all $k\geq0$.

\bfi
\begin{tikzpicture}
\draw(-.5,1) node {$B=$};
\foreach \x in {0} 
   \draw (\x,0)--(\x,.5);
\foreach \x in {.5} 
   \draw (\x,0)--(\x,1);
\foreach \x in {1,1.5}
    \draw (\x,0)--(\x,1.5);
\foreach \x in {2,2.5}
   \draw (\x,0)--(\x,2.5);
\foreach \y in {0,.5}
	\draw (0,\y)--(2.5,\y);
\foreach \y in {1}
        \draw (.5,\y)--(2.5,\y);
\draw (1,1.5)--(2.5,1.5);
\foreach \y in {2,2.5}
          \draw(2,\y)--(2.5,\y);
\draw (.75, .75) node{$R$};
\draw (2.25, 2.25) node{$R$};
\draw (1.25, 1.25) node{$R$};
\foreach \y in {1,2}
   \draw[thick,dashed] (0,\y)--(2.5,\y);
\end{tikzpicture}
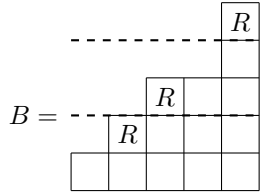
\capt{\label{mrook} A $2$-level rook placement of $3$ rooks on $B= (1,2,3,3,5)$. Thick dashed lines are divisions between levels. Notice that no further rooks could be placed on this board, since each level already contains a rook.}
\efi

Briggs and Remmel also defined an $m$ analogue of the rook polynomial as follows. In order to define this polynomial, we need to define the \emph{$k$th $m$-falling factorial} of $x$, defined by $x\da_{k,m} = \prod_{i=0}^{k-1} x-im$. Note that this is a generalization of the $k$th falling factorial and that $\{x\da_{k,m}\mid k\geq0\}$ is another basis for the vector space of polynomials. Given this definition, we can define the \emph{$m$-level rook polynomial} of board $B$ with $n$ columns by $$p_m(B,x) = \sum_{k=0}^nr_{k,m}(B)x\da_{n-k,m}.$$ This is the generating function for the $m$-level rook numbers of $B$ in the $m$-falling factorial basis of polynomials.

In order to factor the $m$-level rook polynomial more cleanly, Briggs and Remmel restricted their attention to a nice subset of Ferrers boards, called singleton boards. To define these boards, we need to first define the \emph{$m$-floor of $o$}, which we write $\lfloor o \rfloor_m$. The $m$-floor of $o$ is the largest multiple of $m$ less than or equal to $o$.  Similarly, define the \emph{$m$-ceiling of $o$} to be the smallest multiple of $m$ greater than or equal to $o$, written $\lceil o \rceil_m$. A board $B = (b_1, b_2,\ldots,b_n)$ is singleton if $b_i -\lfloor b_i \rfloor_m \neq 0$ implies that $\lfloor b_i \rfloor_m < \lfloor b_{i+1} \rfloor_m$. Alternatively, this means that for each level there is at most a single column of $B$ that intersects that level in more than one cell but less than the full $m$ cells, thus the term ``singleton''. Whether a given board is singleton will depend on the value of $m$ being considered, and when $m=1$ every board is singleton. Figure~\ref{singleton} gives an example of a singleton and a non-singleton board when $m =2$.

\bfi
\begin{tikzpicture}
\draw(-.5,1) node {$B_1=$};
\foreach \x in {0} 
   \draw (\x,0)--(\x,.5);
\foreach \x in {.5} 
   \draw (\x,0)--(\x,1);
\foreach \x in {1,1.5}
    \draw (\x,0)--(\x,1.5);
\foreach \x in {2,2.5}
   \draw (\x,0)--(\x,2.5);
\foreach \y in {0,.5}
	\draw (0,\y)--(2.5,\y);
\foreach \y in {1}
        \draw (.5,\y)--(2.5,\y);
\draw (1,1.5)--(2.5,1.5);
\foreach \y in {2,2.5}
          \draw(2,\y)--(2.5,\y);
\foreach \y in {1,2}
   \draw[thick,dashed] (0,\y)--(2.5,\y);
\end{tikzpicture}
\hs{40pt}
\begin{tikzpicture}
\draw(-.5,1) node {$B_2=$};
\foreach \x in {0} 
   \draw (\x,0)--(\x,.5);
\foreach \x in {.5,1} 
   \draw (\x,0)--(\x,1);
\foreach \x in {1.5}
    \draw (\x,0)--(\x,2);
\foreach \x in {2,2.5}
   \draw (\x,0)--(\x,2.5);
\foreach \y in {0,.5}
	\draw (0,\y)--(2.5,\y);
\foreach \y in {1}
        \draw (.5,\y)--(2.5,\y);
\foreach \y in {1.5,2}
     \draw (1.5,\y)--(2.5,\y);
\foreach \y in {2.5}
          \draw(2,\y)--(2.5,\y);
\foreach \y in {1,2}
   \draw[thick,dashed] (0,\y)--(2.5,\y);
\end{tikzpicture}
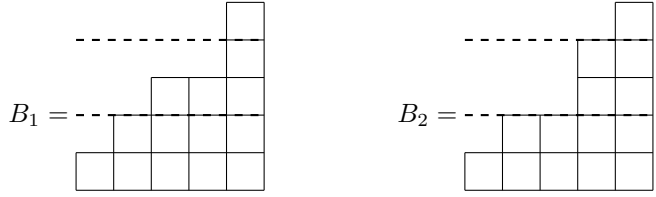
\capt{\label{singleton} $B_1$ is not singleton, because both columns $3$ and $4$ intersect the $2$nd level at some, but not all, squares. $B_2$ is a singleton board.}
\efi

Briggs and Remmel gave the following factorization theorem for the $m$-level rook polynomial of a singleton board.

\begin{thm}[Briggs-Remmel~\cite{br:mrn}]
\label{br factor}
If $B=(b_1,\ldots,b_n)$ is a singleton board, then
$$
p_m(B,x) = \prod_{i=1}^n(x+b_i-m(i-1)).
$$
\end{thm}

Notice that Theorem~\ref{gjw1} is a special case of Theorem~\ref{br factor} when $m=1$. This factorization theorem gives rise to an \emph{$m$-level root vector} for singleton board $B=(b_1,\ldots,b_n)$, defined by $\xi_m(B)=\langle0-b_1,m-b_2,\ldots,m(n-1)-b_n\rangle$. The elements of the $m$-level root vector are once again the zeros of the $m$-level rook polynomial of $B$. Barrese, Loehr, Remmel, and Sagan gave criteria for when a given vector is an $m$-level root vector for some singleton board.

\begin{thm}[\cite{blrs:mrp}]
\label{mvector}
A vector $\xi = \langle\xi_1,\xi_2,\ldots,\xi_n\rangle$ is the $m$-level root vector of some singleton board if and only if
\ben
   \item[(i)] $\xi_1 \leq 0$
   \item[(ii)] $\xi_i +m \geq \xi_{i+1} $
   \item[(iii)] If neither $\xi_i$ nor $\xi_{i+1}$ are multiples of $m$ then $\lfloor \xi_i \rfloor_m \geq \lfloor \xi_{i+1} \rfloor_m$.
\een
\end{thm}

It is worth noting that the condition requiring that $\xi_{i+1}$ be at most $m$ more than $\xi_i$ means that, if $\xi_i$ is a multiple of $m$ and $\lfloor \xi_i \rfloor_m < \lfloor \xi_{i+1} \rfloor_m$, which is allowed under condition (iii) in this case, then $\xi_{i+1}$ must also be a multiple of $m$. Specifically $\xi_{i+1} = \xi_i +m$. Thus increasing the value of $\xi_i$ cannot cause condition (iii) to cease to hold.

The next theorem, from the same paper, extends the factorization theorem to all Ferrers boards by considering the number of squares in each level, rather than each column. Suppose board $B$ has  non-empty intersection with level $n$, but empty intersection with level $n+1$. Then we define $l_j$, the \emph{$j$-th level number} of $B$, to be the number of cells of $B$ that intersect the $j$th level from the top, level $n+1-j$. Using this definition yields another factorization theorem.

\begin{thm}[\cite{blrs:mrp}]
\label{level factor}
If $B=(b_1,\ldots,b_n)$ is a Ferrers board, then
$$
p_m(B,x) = \prod_{j=1}^n (x+l_j-m(j-1)).
$$
\end{thm}

One unfortunate lack of this factorization theorem is that, unlike Theorem~\ref{gjw1} and Theorem~\ref{br factor}, we cannot reconstruct the exact shape of the board uniquely $B$ from the roots of its $m$-level rook polynomial.

\section{Graph Theory of $m$-level Rook Placements}
\label{sec:mgtrp}

We will develop an $m$-level rook equivalence graph in this section. The results are generalizations of the results in Section~\ref{sec:gtrp}.

In this section we restrict our attention to singleton boards. Pragmatically this makes sense as it allows us to use the ($m$-level) root vector as before. Additionally, it does not seem likely that considering all Ferrers boards would introduce interesting complexity, since Theorem~\ref{level factor} implies that any Ferrers board can be transformed into a singleton board by moving squares from one column to another. Indeed, any sort of column rearranging can be done as long as it doesn't alter the level numbers of the board.

We define the \emph{$m$-level rook equivalence graph} of the $m$-level rook equivalence class containing singleton board $B$ as follows. The vertex set consists of all singleton boards $m$-level rook equivalent to $B$. As before, two singleton boards are connected by an edge if one can be transformed into the other by moving some number of squares from one column to another. Let $G_m(B)$ denote the $m$-level rook equivalence graph of the equivalence class containing singleton board $B$. Once again, we quickly replace the geometric notation with one focused on the $m$-level root vector.

\begin{thm}
\label{mrootswap}
If singleton boards $B_1 \equiv B_2$, both containing $n$ columns, are such that $\{B_1,B_2\}\in E$, then the $m$-level root vectors $\xi_m(B_1)$ and $\xi_m(B_2)$ differ only in two positions, which are swapped. That is, if $\xi_m(B_1) = \langle z_1,\ldots,z_{i_1},\ldots,z_{i_2},\ldots,z_n\rangle$ then $\xi_m(B_2)=\langle z_1,\ldots,z_{i_2},\ldots,z_{i_1},\ldots,z_n\rangle$ for some $1\leq i_1 \langle i_2 \leq n$ where $z_{i_1} \neq z_{i_2}$.
\end{thm}

\begin{proof}
The proof is exactly the same as that of Theorem~\ref{rootswap}, except the $i$th element of the $m$-level root vector is $(i-1)m-b_i$ instead of $(i-1)-b_i$ as in the ordinary root vector.
\end{proof}

\begin{lem}
\label{mchecktwice}
Suppose $\xi_m(B)=\langle z_1,\cdot,z_n\rangle$ is the $m$-level root vector for some singleton board. Let $\xi'$ be equivalent to $\xi(B)$ except in positions $1<i_1$ and $1<i_2$ which are swapped, and further assume that $z_{i_1}>z_{i_2}$. If $z_{i_2-1} \geq z_{i_1}$ and $z_{i_2} \geq z_{i_1+1}$ then $\xi'$ is the $m$-level root vector for some singleton board.
\end{lem}

\begin{proof}
Since conditions (ii) and (iii) of Theorem~\ref{mvector} limit how much elements of the $m$-level root vector can grow from one to the next, decreasing the value in position $i_1$ guarantees that if the element initially there was not too much larger than the element in position $i_1-1$, then the new element also cannot be. Similarly, we increase the value in position $i_2$, so if the element in position $i_2+1$ was not too much larger than the original element in position $i_2$, it will continue to not be too much larger. It remains to check that the new element in position $i_1$ is not too small compared to the element in position $i_1+1$ and the new element in position $i_2$ is not too much larger than the element in position $i_2-1$.

Clearly if $a\geq b$ then $a+m \geq b$ and $\lfloor a \rfloor_m \geq \lfloor b \rfloor_m$. Thus if $z_{i_2-1} \geq z_{i_1}$ and $z_{i_2} \geq z_{i_1+1}$ conditions (ii) and (iii) hold for the sets of neighbors we must check, and $\xi'$ fulfills the conditions of Theorem~\ref{mvector} and is therefore the $m$-level root vector of some singleton board.
\end{proof}

\begin{lem}
\label{mdecreasing}
Let $B$ be a singleton board with $n$ columns that has $M$ as the maximum element in its $m$-level root vector. If the first appearance of $M$ in $\xi_m(B)$ is in position $j$, then vertex $B$ is connected by a sequence of edges in $G_m(B)$ to another Ferrers board $B'$ which has an $m$-level root vector $\xi_m(B')$ identical to $\xi_m(B)$ in the first $j$ positions, and weakly decreasing for indices $i \geq j$.
\end{lem}

\begin{proof}
This proof closely mirrors that of Lemma~\ref{decreasing}. Remember the general idea of the proof of Lemma~\ref{decreasing} was to take the leftmost element to the right of position $j$ that was not already bigger than all elements to its right, and swap it by the rightmost element of the root vector that was bigger than it. Let us assume the element in question is in position $i$, and it is getting swapped for a larger element in position $k$. Since $z_i<z_k$, by Lemma~\ref{mchecktwice} if $z_{i-1}\geq z_k$ and $z_{i} \geq z_{k+1}$, we can conclude that the resulting vector is an $m$-level root vector.Since the element in position $i-1$ was bigger than every element to its right, it will be bigger than the new element in position $i$, so $z_{i-1} \geq z_k$. Also, since $z_i$ is swapped for the rightmost element of the $m$-level root vector greater than it, $z_i \geq z_{k+1}$. Thus Lemma~\ref{mchecktwice} guarantees that conditions (ii) and (iii) of Theorem~\ref{mvector} are true for the new vector, so it is the $m$-level root vector of a new board, which shares an edge with the original board $B$.

Therefore we can iterate swaps, as in the proof of Lemma~\ref{decreasing}, until we obtain an $m$-level root vector where all elements including and to the right of $z_j=M$ are in non-increasing. The board we obtain after these sequence of swaps will be connected back to the original board $B$ by a sequence of edges in $G_m(B)$ which correspond to the swaps performed.
\end{proof}

Having demonstrated the $m$-level version of Lemma~\ref{decreasing}, we move on to the $m$-level version of Lemma~\ref{maxleft}

\begin{lem}
\label{mmaxleft}
Let $B$ be a singleton board with $n$ columns, padded with columns of height zero on the left so that $\xi_m(B)=\langle z_1,\ldots,z_n\rangle$ contains only non-negative entries. Let $M$ be the greatest integer such that $v_M \neq 0$. Let $S$ be the set of non-negative multiples of $m$ which are strictly less than $M$. Let $i$ be the least non-negative integer such that either $v_i>1$ or $v_i=1$ and $i$ is not a multiple of $m$. If $v_s \geq 2$ for all $s\in S$ such that $s>i$, then $B$ is connected by a sequence of edges in $G_m(B)$ to another Ferrers board $B'$ which has a root vector $\xi_m(B') = \langle0,m,\ldots,\lfloor M-1\rfloor_m,M,\ldots,z_n\rangle$. That is to say, the first $\frac{\lfloor M-1\rfloor_m}{m}+1$ entries of $\xi(B')$ are the non-negative multiples of $m$ from $0$ to $\lfloor M-1\rfloor_m$ in increasing order and the next element is the first occurrence of value $M$.
\end{lem}

\begin{proof}
Apply the algorithm from the proof of Lemma~\ref{mmaxleft} to the elements of $\xi_m(B)$ which are multiples of $m$ and occur twice to the left of the first occurrence of value $M$. This yields a board where each multiple of $m$ occurs only once to the left of the leftmost entry $M$. By Lemma~\ref{mdecreasing} we can swap elements until the elements to the right of the first occurrence of $M$ are in weakly decreasing order. However, there may be entries that are not multiples of $m$ interspersed with the increasing sequence of multiples of $m$ on the left. Since each multiple of $m$ occurs only once on that side of $M$, each instance of a non-multiple of $m$ must happen between the unique occurrence of the greatest multiple of $m$ lower than it, its $m$-floor, and the least multiple of $m$ greater than it, its $m$-ceiling.

Let $i$ be the index of the leftmost entry in the new $m$-level root vector that is not a multiple of $m$ and $v$ be the value of that entry. Swap $v$ with the rightmost occurrence of $\lceil v \rceil_m$ in the weakly decreasing section of the root vector to the right of $M$, or just swap with $M$ if $\lceil v \rceil_m \geq M$. If $\lceil v \rceil_m < M$, an entry with that value will exist to the right of $M$, because any multiple of $m$ greater than $v$ and less than $M$ must occur at least twice in the root vector, and only once to the left of $M$. The resulting vector is a root vector for a singleton board because $v$ was swapped into the weakly decreasing section of the root vector and was swapped with the last element that has an $m$-floor greater then it. Furthermore, $\lfloor v \rfloor_m < z_{i+1} \leq \lceil v \rceil_m$ so the new element in position $i$ which has value $\lceil v \rceil_m$ will be greater than or equal to the value to its right. 

Next swap the current entry in position $i+1$ with $\lceil v \rceil_m+m$, the next multiple of $m$ bigger than $\lceil v \rceil$, or with $M$ if $\lceil v \rceil_m+m \geq M$. This is possible for the same reasons provided in the previous paragraph. Continuing this way, we obtain a root vector which begins $\langle0,m,\ldots,\lfloor M-1\rfloor_m,M,\ldots,\rangle$ as desired.
\end{proof}

As before, this allows us to characterize when the $m$-level rook equivalence graph of a board is connected.
\begin{thm}
\label{mconnected}
Let $B$ be a singleton board with $n$ columns, padded with columns of height zero on the left so that $\xi_m(B)=\langle z_1,\ldots,z_n\rangle$ contains only non-negative entries. Let $M$ be the greatest integer such that $v_M \neq 0$. Let $S$ be the set of non-negative multiples of $m$ which are strictly less than $M$. Let $i$ be the least non-negative integer such that either $v_i>1$ or both $v_i=1$ and $i$ is not a multiple of $m$. In this case $G_m(B)$ is connected if and only if $v_s \geq 2$ for all $s\in S$ such that $s>i$.
\end{thm}

\begin{proof}
Lemma~\ref{mmaxleft} demonstrates that vertex $B$ is connected by a sequence of edges to a vertex corresponding to a singleton board that has an $m$-level rook polynomial of the form $ \langle0,m,\ldots,\lfloor M-1\rfloor_m,M,\ldots,z_n\rangle$. Applying Lemma~\ref{mdecreasing} to this second board yields a board which increases by multiples of $m$ until $\lfloor M-1\rfloor_m$, then is weakly decreasing starting with the value of $M$ in the next position. Such a root vector must be unique in a given $m$-level rook equivalence class. In fact, in~\cite{blrs:mrp} it is shown that this $m$-level root vector corresponds to a unique board in the $m$-level equivalence class where each non-empty column is at least $m$ squares shorter than the column to its right.

To complete the proof, we need to show that if some multiple of $m$, which we will call $v\neq M$, only occurs once in the root vector, but there is a multiple of $m$ less than it that occurs twice, or any non-multiple of $m$ less than it that occurs once, then $G_m(B)$ cannot be connected. The proof begins by showing that, $v$ can never be swapped with another value. Since $v\neq M$, there is at least one element of $\xi_m(B)$ which is bigger than $v$. Because of property (iii) of $m$-level root vectors, $\lfloor \xi_i \rfloor_m < \lfloor \xi_{i+1}\rfloor_m$ implies that $\xi_{i+1}$ must be a multiple of $m$. So, the only way to go from a value less than $v$ to one greater than $v$ is to have value $v$ in between, since we are assuming $v$ is a multiple of $m$. Therefore the entry to the right of $v$ must be the the leftmost entry of $\xi_m(B)$ that is bigger than $v$, because there is only one entry in $\xi_m(B)$ with value $v$. Thus $v$ cannot be swapped with another entry of the $m$-level root vector, because it would no longer be directly to the left of the leftmost value in the root vector greater than itself. Nor can the leftmost value in the root vector greater than $v$ ever move, because then it would no longer be directly to the right of the element of value $v$. Note that this does not preclude swapping that element with another element, so long as that element is no greater than $v+m$.

To get two root vectors with $v$ in different positions, consider the unique root vector which increases by multiples of $m$ until it reaches $M$, then is weakly decreasing. In this root vector, $v$ must be in position $\frac{v}{m}+1$. On the other hand, there is is a value less than $v$ which occurs after $M$ in this unique ordering. Take that value, and insert to the left of $v$, so the entries to the left of $v$ are still weakly increasing. Since the values were increasing by exactly $m$ before, they will now increase by at most $m$ from one value to the next, so this new vector is the $m$-level root vector of some singleton board, but value $v$ occurs exactly one position later in this new vector than in the old. Since the position of the element with value $v$ cannot be altered by swapping pairs of entries, the singleton board corresponding to the new vector cannot be connected by a sequence of edges in $G_m(B)$ to the singleton board corresponding to the original vector. Therefore $G_m(B)$ is disconnected.
\end{proof}

\section{Future Projects}
\label{sec:fproj}

This final section puts forward two conjectures related to the work in the previous sections. The first deals with what graphs can arise as rook equivalence graphs of boards and the second is an attempt to use the rook equivalence board to provide explicit, geometrically motivated bijections between rook placements on rook equivalent boards when $G(B)$ is connected.

\begin{conj}
If $K_{a,b}$ is a complete, bipartite graph that is the rook equivalence graph of Ferrers board $B$, then $a=b=1$.
\end{conj}

As we saw in the proofs of Theorems~\ref{no22} and~\ref{no33}, cycles with odd length show up frequently in the rook equivalence graphs. In fact, if value $v_1$ occurs in two different positions of $\xi(B)$ which can both be exchanged with an entry with another value $v_2$, this will induce a cycle of length three in the rook equivalence graph of $B$. It may be possible to provide an argument why any rook equivalence graph on at least three vertices must have an odd length cycle, or produce a $K_{a,b} = G(B)$ to act as a counterexample.

\begin{conj}
If $\{B_1,B_2\}$ is an edge in the rook equivalence graph $G(B)$, then there is an explicit bijection between placements of $k$ rooks on $B_1$ and placements of $k$ rooks on $B_2$.
\end{conj}

In~\cite{fs:rpf} Foata and Sch\"{u}tzenberger give explicit bijections between any two rook equivalent boards by transposing well chosen subboards of those boards. The same is done for $m$-level rook equivalent singleton boards in~\cite{blrs:bom}. Unfortunately, transposing part of a board is a geometric change that is hard to understand in terms of the root vector. On the other hand, moving cells from one column to another is a much simpler geometric transformation, and it is easy to understand in terms of the root vector. However, it has been unexpectedly difficult to extend moving at most one rook from one column to another to an explicit bijection taking the rook placements on the first board to rook placements on the second board.

This is surprising, given how much geometric structure there is to the situation. If $z_i$ and $z_j$ are the two entries of the root vector that are swapped, $|z_i-z_j|$ squares get moved from the column whose entry gets smaller to the column whose entry increases. Furthermore, if a square moves to the right $k$ columns, it will also move up $k$ rows, equivalently for left and down. This is because, before the exchange, $(i-1)-b_i = z_i$ and after the swap $(j-1) - b'_j = z_i$, so $(i-1)-b_i = (j-1) - b'_j$ and $b'_j-b_i = j-i$. Given all this geometric structure, it seems like it must be possible to construct a geometrically motivated bijection between rook placements on the two boards.

\bibliographystyle{alpha}
\bibliography{kenny}
\label{sec:biblio}

\end{document}